\documentclass{amsart}
\usepackage{amsmath, amssymb, latexsym}
\title{Higher dimensional Moore bounds}
\author{Michael Goff}

\newtheorem{theorem}{Theorem}[section]

\newtheorem{lemma}[theorem]{Lemma}

\newtheorem{conjecture}[theorem]{Conjecture}

\newtheorem{problem}[theorem]{Problem}
\newtheorem{claim}[theorem]{Claim}

\newcommand{\K}{\Gamma}
\newcommand{\field}{{\bf k}}
\newcommand{\codim}{\mbox{\upshape codim}\,}

\newcommand{\Skel}{\mbox{\upshape Skel}\,}
\newcommand{\lk}{\mbox{\upshape lk}\,}
\newcommand{\g}{\mbox{\upshape gr}\,}

\newcommand{\avg}{\mbox{\upshape avg}\,}
\newcommand{\cp}{\mbox{\upshape oc}\,}
\newcommand{\CP}{\mbox{\upshape OC}\,}

\def\proofof#1{\smallskip\noindent {\it Proof of #1: \ }}
\def\endproof{\hfill$\square$\medskip}

\begin{document}

\begin{abstract}
We prove upper bounds on the face numbers of simplicial complexes in terms on their girths, in analogy with the Moore bound from graph theory.  Our definition of girth generalizes the usual definition for graphs.
\end{abstract}

\date{June 3, 2009}

\maketitle

\section{Introduction}
The Moore bound in graph theory answers the following classical question.  What is the maximum number of edges in a graph with $n$ vertices and no cycles with $g$ or fewer vertices?  Phrased differently, the Moore bound gives the fewest number of vertices in a graph with \textit{girth} (that is, the length of the shortest cycle) greater than $g$ and average degree $a$.

\begin{theorem}
\label{IrregularMoore}
\cite{MooreGraphs} Let $G$ be a graph with average degree $a \geq 2$ and girth greater than $g$.  Then $G$ has at least $n=n_0(a,g+1)$ vertices, where $$n_0(a,2r) = 2 \sum_{i=0}^{r-1} (a-1)^i,$$ $$n_0(a,2r+1) = 1+a\sum_{i=0}^{r-1}(a-1)^i.$$
\end{theorem}

Theorem \ref{IrregularMoore} answers an old problem, which appears in \cite[Problem 10, p. 163]{Bollobas}.  A relatively simple proof for $a$-regular graphs is found in \cite{Biggs}, and a weaker inquality is proven in \cite{WeakerMoore}.  Theorem \ref{IrregularMoore} was proven in \cite{MooreGraphs} using random walks on the graph.

In this paper we consider similar bounds for simplicial complexes.  A \textit{simplicial complex} $\K$ with the vertex set $V(\K) = V = \{x_1,\ldots,x_n\}$ is a collection of subsets of $2^{V}$ called \textit{faces} such that $\K$ is closed under inclusion.  The \textit{dimension} of $\K$ is one less than the maximum cardinality of a face of $\K$.  If $W \subseteq V$, the \textit{induced subcomplex} of $\K$ on $W$, denoted $\K[W]$, has vertex set $W$ and faces $\{F: F \in \K, F \subseteq W\}$.  The \textit{face numbers} are given by $f_i(\K)$, which denotes the number of faces with $i+1$ vertices in $\K$.  For $F \subset V(\K)$, the \textit{link} of $F$, denoted $\lk_\K(F)$, is the simplicial complex that has vertex set $V-F$ and faces $\{G-F: F \subset G \in \K\}$.

Fix a base field $\field$.  The $i$-th reduced simplicial homology of a simplicial complex $\K$ with coefficients in $\field$ is denoted by $\tilde{H}_i(\K;\field)$.  We define the $(p-1)$-\textit{girth} of a simplicial complex $\K$ by $$\g_{p-1}(\K) := \min\{|W|: \tilde{H}_{p-1}(\lk_\K(F)[W];\field) \neq 0 \ \mbox{ for some } \ \emptyset \subseteq F \in \K\},$$ or $\infty$ is no such $W$ exists.  Although the value of $\g_{p-1}(\K)$ may depend on $\field$, our theorems hold regardless of which field is chosen.  Another paper \cite{MooreSC} proves an analogue of the Moore bound for simplicial complexes, but uses a different definition of girth.

In words, $\g_{p-1}(\K)$ is the fewest number of vertices in a subcomplex in $\K$ that has a nonzero $(p-1)$-cycle in homology, where we consider subcomplexes that are induced in links of faces.  When $\dim \K = 1$, i.e. $\K$ is a graph, our definition of $\g_1$ reduces to the usual definition of girth regardless of $\field$.  In that case, $\g_1(\K) = \infty$ if $\K$ is a forest, and $\g_1(\K)$ is otherwise the number of vertices in a shortest cycle of $\K$.

We note some properties of the girth.  The following is immediate from the definition.

\begin{lemma}
\label{GirthProperties}
Let $\K$ be a simplicial complex.  The following inequalities hold for all $p$. \newline
1) For all $W \subset V(\K)$, $\g_{p-1}(\K) \leq \g_{p-1}(\K[W])$. \newline
2) For all $F \in \K$, $\g_{p-1}(\K) \leq \g_{p-1}(\lk_\K(F))$.
\end{lemma}

Our main results are as follows.  In Section \ref{OneGirth}, we prove an upper bound on the number of edges in terms of the one-girth: if $\K$ has $n$ vertices and dimension $d-1$, $\g_1(\K) > 2r$, and the quantity $r/\log(n/d)$ sufficiently small,  then $$f_1(\K) \leq (2^{-1}+\epsilon)(d-1)^{1-1/r}n^{1+1/r}$$ for an arbitrary $\epsilon > 0$, whereas if $\g_1(\K) > 2r+1$, then $$f_1(\K) \leq (2^{-1-1/r}+\epsilon)(d-1)^{1-1/r}n^{1+1/r}.$$  In Section \ref{OneGirthLargeI}, we prove that if $\g_1(\K) > 2r$, then for some constant $C_{r,i}$ that depends only on $r$ and $i$, $$f_i(\K) \leq C_{r,i}d^{1-1/r-1/r^2- \ldots - 1/r^i}n^{1+1/r+1/r^2+ \ldots + 1/r^i}.$$  In Section \ref{HighGirth}, we conjecture a general upper bound on $f_i$ when $\g_{p-1}(\K)$ is given, and we prove that conjecture in some special cases.  In Section \ref{Existence}, we establish the existence of some simplicial complexes with high girth and large face numbers using probabilistic methods.

\section{One-girth and the number of edges}
\label{OneGirth}
In this section we prove an upper bound on the number of edges of a simplicial complex when the $1$-girth is given.  The following is the main theorem of the section.

\begin{theorem}
\label{FixedGirthStrong}
Let $\K$ be a $(d-1)$-dimensional simplicial complex with $n$ vertices and $\g_1(\K) > 2r$, $r \geq 2$.  For every $\epsilon > 0$, there exists $\delta$ such that if $r/\log (n/d) < \delta$, then $$f_1(\K) \leq (2^{-1}+\epsilon)(d-1)^{1-1/r}n^{1+1/r}.$$
\noindent
Furthermore, if $\g_1(\K) > 2r+1$, then $$f_1(\K) \leq (2^{-1-1/r}+\epsilon)(d-1)^{1-1/r}n^{1+1/r}.$$
\end{theorem}
In the case that $d=2$, the upper bound on $f_1$ of Theorem \ref{FixedGirthStrong} is approximately equal to that of Theorem \ref{IrregularMoore} for values of $r$ small relative to $\log(n)$.  Our proof uses some of the same techniques used in \cite{MooreGraphs} to prove Theorem \ref{IrregularMoore}.

To prove theorem \ref{FixedGirthStrong} we introduce flag complexes.  We say that a simplicial complex $\K$ is \textit{flag} if all the minimal non-faces of $\K$ consist of two vertices, or equivalently if $F$ is a face of $\K$ whenever all the $2$-subsets of $F$ are faces.  A flag complex is also called a \textit{clique complex}.  We establish some properties of girths of flag complexes.  The second property allows us to assume that $\K$ is flag in the proof of Theorem \ref{FixedGirthStrong}.

\begin{lemma}
\label{FlagProperties}
Let $\K$ be a simplicial complex. Then the following hold. \newline
1) Let $F \in \K$ and $W \subset V(\K)$ so that $F \cap W = \emptyset$ and $F \cup \{w\}$ is a face in $\K$ for all $w \in W$.  If $\K$ is flag, then $\K[W] = \lk_\K(F)[W]$. \newline
2) $\K$ is flag if and only if $\g_1(\K) \geq 4$. \newline
3) If $\K$ is flag and $\g_{p-1}(\K) < \infty$, then there exists $W \subset V(\K)$ such that $|W|=\g_{p-1}(\K)$ and $\tilde{H}_{p-1}(\K[W]);\field) \neq 0$.
\end{lemma}
\begin{proof}
Suppose the conditions of the first claim hold, and let $F'$ be a face of $\K[W]$.  The conditions imply that there is an edge $uv$ for all $u,v \in F \cup F'$, and so $F \cup F'$ is a face in $\K$.  Then $F' \in \lk_\K(F)$.  Also, every face of $\lk_\K(F)$ is a face in $\K$, and this proves the first claim.  The third claim is immediate from the first.

To prove the second claim, first suppose that $\K$ is flag.  Then the link of every face is also flag by the first claim, and so there is no $F \in \K$ and $W \subset V(\K)$ so that $|W|=3$ and $\lk_\K[W]$ is exactly a graph-theoretic $3$-cycle.  Hence $\g_1(\K) \geq 4$.  Now suppose that $\K$ is not flag, and let $W$ be a minimal non-face of $\K$ with $|W| \geq 3$.  Choose $W' \subset W$ with $|W'| = |W|-3$.  Then $\lk_\K(W')[W-W']$ is a $3$-cycle, and so $\g_1(\K) = 3$.
\end{proof}

The proof of Theorem \ref{FixedGirthStrong} requires several technical lemmas.  The first is a condition on when, given that there exists a graph-theoretic cycle in a simplicial complex on vertices $v_1,\ldots,v_r$, we can conclude that $\g_1(\K) \leq r$.
\begin{lemma}
\label{CycleLemma}
Let $\K$ be a simplicial complex containing a graph theoretic cycle with (not necessarily distinct) vertices $v_1, \ldots, v_r$ and edges $v_i v_{i+1}$ for $1 \leq i \leq r$ (subscripts are mod $r$).  Suppose that there exists at most one value of $i$ such that $\{v_{i-1},v_i,v_{i+1}\}$ is a face in $\K$.  Then $\g_1(\K) \leq r$.
\end{lemma}
\begin{proof}
In the case that $r=3$, the conditions imply that $\K[v_1,v_2,v_3]$ is the boundary of a triangle, and so $\g_1(\K) = 3$.  Assume that $r \geq 4$.  If for some $i$, $v_{i-1} \neq v_{i+1}$, and there exists an edge $v_{i-1}v_{i+1}$ but no triangle $\{v_{i-1},v_i,v_{i+1}\}$, then $\g_1(\K)=3$.  Therefore, we may assume this condition: suppose that there exists at most one value of $i$ such that either $v_{i-1}=v_{i+1}$ or there is an edge $v_{i-1}v_{i+1}$.  Assume without loss of generality that if such an $i$ exists, $i=2$.  If the $v_i$ are distinct and $\K[v_1,\ldots,v_r]$ contains no edges except each $v_iv_{i+1}$, then $\g_1(\K) \leq \g_1(\K[v_1,\ldots,v_r]) = r$ and the lemma is true.  Otherwise, we may choose $j$ and $k$ so that $k-j$ is minimal, subject to the following conditions: $k \geq j+2$, $v_jv_k$ is an edge in $\K$, and $(j,k) \neq (1,3)$.  Since for all $i' \neq 2$, $v_{i'-1} \neq v_{i'+1}$ and $v_{i'-1}v_{i'+1}$ is not an edge in $\K$, such $j$ and $k$ can always be chosen so that $k \geq j+3$.  Then $\K[v_j, \ldots, v_k]$ is a graph theoretic cycle and the lemma holds.
\end{proof}

The next lemma roughly states that if $\K$ is flag and $\g_{p-1}(\K) > 2p$, then $\K$ does not have too many edges.  Define the $i$-\textit{skeleton} of $\K$, denoted $\Skel_i(\K)$, to be the simplicial complex with vertex set $V(\K)$ and faces $\{F: F \in \K, |F| \leq i+1\}$.  For $v \in V(\K)$, $\deg v$ denotes the number of edges that contain $v$.

\begin{lemma}
\label{HighDegLemma}
Let $p$ be fixed, and let $\K$ be a flag $(d-1)$-dimensional simplicial complex with $n$ vertices, and suppose that $\g_{p-1}(\K) > 2p$ and $d < (1-\delta)n$ for some $\delta>0$.  Then there exists an $\epsilon>0$, which depends only on $\delta$ and $p$, such that $f_1(\K) < {n \choose 2} - \epsilon n^2$.  Furthermore, in the case that $p=2$, for every $\epsilon' > 0$ there exists $\delta' > 0$ such that if $d < \delta' n$, then $f_1(\K) < \epsilon' n^2$.
\end{lemma}
\begin{proof}
We prove the first statement by induction on $p$.  In the case that $p=1$, $d < n$ implies that $\K$ is not a simplex, which implies that $\g_0(\K) = 2$; hence the $p=1$ case is empty.  Let $\delta$ be given, and suppose by way of contradiction that for arbitrarily small $\epsilon$, there exists a $(d-1)$-dimensional simplicial complex $\K$ with $n$ vertices satisfying $d < (1-\delta)n$ and $f_1(\K) \geq {n \choose 2}-\epsilon n^2$.  There exist $\epsilon_1, \epsilon_2 > 0$ with $\epsilon_1 \rightarrow 0$ and $\epsilon_2 \rightarrow 0$ as $\epsilon \rightarrow 0$ such that $\K$ contains a set of vertices $Y$ such that $|Y| \geq (1-\epsilon_2)n$ and every $v \in Y$ has at least $(1-\epsilon_1)n$ neighbors; otherwise there would be at least $\epsilon_1 \epsilon_2 n^2/2$ pairs of vertices not joined by an edge, a contradiction.  Since $\dim \K < (1 - \delta)n$ and $\K$ is flag, then if $\epsilon_2 < \delta$, there exists $u,v \in Y$ that are not adjacent.  Let $W$ be the set of vertices adjacent to both $u$ and $v$.  Then $|W| \geq (1-2\epsilon_1)n$, and so $f_1(\K[W]) \geq {n \choose 2}-\epsilon n^2-2\epsilon_1 n^2$.  By choosing $\epsilon$ and $\epsilon_1$ sufficiently small, it follows by the inductive hypothesis that $\g_{p-2}(\K[W]) \leq 2p-2$.  By Part 3 of Lemma \ref{FlagProperties}, we may choose $W' \subset W$ so that $\tilde{H}_{p-2}(\K[W'];\field) \neq 0$ and $|W'| \leq 2p-2$.  Since $\K[W',u,v]$ is the suspension of $\K[W']$, $\tilde{H}_{p-1}(\K[W',u,v],\field) \neq 0$ and we conclude that $\g_{p-1}(\K) \leq 2p$.

Now suppose that for some fixed $\epsilon'$ and for arbitrarily small $\delta'$, there exists a simplicial complex $\K$ with $n$ vertices, dimension $\delta' n$, and $f_1(\K) \geq \epsilon' n^2$, and we derive a contradiction to $\g_{1}(\K) > 4$.  Assume that $\delta' < \epsilon' / 4$.  If there is a vertex $v$ of $\K$ such that $\deg v < (\epsilon'/2)n$, delete $v$ from $\K$, and repeat this operation until the resulting simplicial complex $\K'$ contains no such vertex.  Then $f_1(\K') \geq (\epsilon'/2) n^2$, and every vertex of $\K'$ has degree at least $(\epsilon'/2)n$.  Choose $v$ to be a lowest degree vertex of $\K'$, and let $a := \deg v$.  There are at least $a^2-a$ paths of length $2$ in $\Skel_1(\K')$ with starting vertex $v$ and ending vertex not $v$.  Consider two cases.

Case 1: There are at least $a^2-a-\delta' n^2$ paths of length $2$ starting at $v$ and ending at a neighbor of $v$; call this set of paths $P$.  Then there are at least $(a^2-a-\delta' n^2)/2 = {a \choose 2} - (\delta'/2)n^2$ edges in $\lk_{\K'}(v)$, since every path in $P$ contains an edge in $\lk_{\K'}(v)$, and every such edge is contained in two paths in $P$.  Since $f_0(\lk_{\K'}(v)) = a \geq \epsilon' n/2 > 2\dim \K$, it follows that $\g_1(\lk_{\K'}(v)) \leq 4$ if $\delta'$ is chosen sufficiently small, by the first part of the lemma.  This implies that $\g_1(\K) \leq 4$ by the two parts of Lemma \ref{GirthProperties}.

Case 2: There are fewer than $a^2-a-\delta' n^2$ paths of length $2$ starting at $v$ and ending at a neighbor of $v$.  Then there are more than $\delta' n^2$ paths of length $2$ starting at $v$ and ending at vertices that are neither neighbors of $v$ nor $v$ itself.  Hence there exists a vertex $u \neq v$ such that $u$ is not a neighbor of $v$, and there are $s > \delta'n$ paths of length $2$ starting at $v$ and ending at $u$.  Label those paths $(v,v_1,u), \ldots (v,v_s,u)$.  Since $\dim \K' < s$ and $\K'$ is flag, there exist $i \neq j$ such that $v_i$ and $v_j$ are not neighbors in $\K'$.  Then $\tilde{H}_1(\K'[v,u,v_i,v_j];\field) \neq 0$ and hence $\g_1(\K') \leq 4$.  We conclude that $\g_1(\K) \leq 4$, which proves the lemma.
\end{proof}

\begin{lemma}
\label{HighDegCor}
Let $\K$ be a simplicial complex with dimension $d-1$, $n$ vertices, and $\g_1(\K) > 4$.  For every $\epsilon_1, \epsilon_2 > 0$, there exists $\delta$ such that if $d < \delta n$, then $\K$ contains at most $\epsilon_1 n$ vertices that each have degree at least $\epsilon_2 n$.
\end{lemma}
\begin{proof}
The result follows from the second part of Lemma \ref{HighDegLemma}.
\end{proof}

Our proof of Theorem \ref{FixedGirthStrong} uses a variation of the non-returning walk on $\vec{\Skel}_1(\K)$ that was introduced in \cite{MooreGraphs}.  Here $\vec{\Skel}_1(\K)$ is the directed graph with vertex set $V(\K)$ and directed edges $\vec{uv}$ and $\vec{vu}$ whenever $uv$ is an edge in $\K$.  Let $Q = \{\vec{v_0v_1}, \vec{v_1v_2}, \ldots, \vec{v_{k-1}v_k}\}$ be a path on $\vec{\Skel}_1(\K)$, which we define by its edges.  We say that that $Q$ is a \textit{non-returning walk of length} $k$ if for all $i$, $v_i \neq v_{i+2}$ and $\{v_i,v_{i+1},v_{i+2}\}$ is not a face in $\K$.

\begin{lemma}
\label{TreeLemma}
Let $\K$ be a $(d-1)$-dimensional simplicial complex satisfying $\g_1(\K) > 2r$.  Then there are at most $(d-1)^{r-1}$ non-returning walks of length $r$ between two given vertices of $\vec{\Skel}_1(\K)$.  Furthermore, if $\g_1(\K) > 2r+1$, then a non-returning walk of length $r+1$ starting with the directed edge $\vec{uv}$ and another of length $r+1$ starting with $\vec{vu}$ have different endpoints.
\end{lemma}
\begin{proof}
The first statement is clear for $r=1$, and we use induction on $r$.  Consider non-returning walks of length $r$ between vertices $u$ and $v$, and suppose by way of contradiction that there are $d$ vertices $v_1, \ldots, v_d$ such that there exists a non-returning walk $P_i = (\vec{uu_{i,1}},\ldots, \vec{u_{i,r-2}v_i}, \vec{v_iv})$ for $1 \leq i \leq d$.  Since $r \geq 2$, $\K$ is flag and therefore has no $(d+1)$-clique.  For some $1 \leq i < j \leq d$, there exist $v_i$ and $v_j$ that are not joined by an edge.  The cycle $C = P_i (P_j)^{-1}$, which is constructed by traversing $P_i$ and then $P_j$ in reverse, satisfies the conditions of Lemma \ref{CycleLemma}, which is a contradiction to $\g_1(\K) > 2r$.  It follows that there exist at most $d-1$ vertices $v_1, \ldots, v_{d-1}$ as above.  By the inductive hypothesis, there are at most $(d-1)^{r-2}$ non-returning walks of length $r-1$ from $u$ to each of the $v_i$, and the first statement follows.

To prove the second statement, suppose that there exist two non-returning walks $\vec{uv}P_1$ and $\vec{vu}P_2$, each of length $r+1$, that end at the same vertex.  Then the cycle $\vec{uv}P_1(P_2)^{-1}$ satisfies the conditions of Lemma \ref{CycleLemma}, and so $\g_1(\K) \leq 2r+1$, a contradiction.
\end{proof}

\proofof{Theorem \ref{FixedGirthStrong}}
Our proof is an adaptation of the proof of the main theorem of \cite{MooreGraphs}.  Let $\epsilon$ be given, and suppose $a$ is the average degree of a vertex in $\K$.  If $M$ is a value that depends only on $\epsilon$, then we may assume that $a > Md$ by choosing $\delta < 1/(\log 2M)$.  We prove the following variant, which implies the theorem: $$n > \frac{((1-\epsilon)a)^r}{(d-1)^{r-1}} \quad \mbox{if} \quad \g_1(\K) > 2r \quad \mbox{and}$$ $$n > \frac{2((1-\epsilon)a)^r}{(d-1)^{r-1}} \quad \mbox{if} \quad \g_1(\K) > 2r+1.$$

If a vertex $v$ satisfies $\deg v < a/2$, then $\K[V(\K)-\{v\}]$ has a higher average degree than $\K$.  Also, $\g_1(\K[V(\K)-\{v\}]) \geq \g_1(\K)$ by Part 1 of Lemma \ref{GirthProperties}.  By considering $\K[V(\K)-\{v\}]$ instead of $\K$, we may assume without loss of generality that all vertices of $\K$ have degree at least $a/2$.

We consider random non-returning walks on $\vec{\Skel}_1(\K)$.  First we specify which edges can be used for those walks.  Fix $\alpha > 0$ so that $\alpha$ depends only on $\epsilon$.  Define $U'$ to be the set of all directed edges $\vec{uv}$ such that either $f_0(\lk_\K(uv)) \geq \alpha f_0(\lk_\K(u))$ or $f_0(\lk_\K(uv)) \geq \alpha f_0(\lk_\K(v))$.  By applying Lemma \ref{HighDegCor} to links of vertices and then the first part of Lemma \ref{FlagProperties}, we conclude that $|U'| < \alpha_1 an$, where $\alpha_1$ can be chosen arbitrarily small by choosing $M$ sufficiently large.

Next set $U := U'$.  If there exists a vertex $v$ such that more than $(4/3) f_0(\lk_\K(v))$ directed edges incident to $v$ are in $U$, then add all directed edges incident to $v$ to $U$.  Repeat this process until no more directed edges are added to $U$ in this way.

We show that $|U| \leq 3\alpha_1 an$.  For any set of directed edges $X$ of $\vec{\Skel}_1(\K)$ and $\vec{uv} \in X$, define the quantity $k(X,\vec{uv})$ to be $1$ is neither $u$ nor $v$ is incident to a directed edge not in $X$, $2$ if exactly one of $u$ and $v$ is adjacent to a directed edge not in $X$, and $3$ otherwise.  Then define $$K(X) := \sum_{\vec{uv} \in X}k(X,\vec{uv}).$$  Note that $K(U') \leq 3\alpha_1 an$, and for any set of directed edges $X$, $K(X) \geq |X|$.  Also, $K(U)$ does not increase at any step in the construction of $U$.  To see that, consider the operation of adding all directed edges incident to $v$ to $U$.  At most $(2/3) f_0(\lk_\K(v))$ directed edges are added, each has $k$-value at most $2$, and at least $(4/3)f_0(\lk_\K(v))$ directed edges have their $k$-values decreased by $1$.  It follows that $|U| \leq 3\alpha_1 an$.  Let $E$ be the set of directed edges of $\vec{\Skel}_1(\K)$ that are not in $U$; $E$ is the set of directed edges that we allow to be used in our random non-returning walks.  By construction, if $\vec{uv} \in E$, then $\vec{vu} \in E$.

For vertices $u,v$ such that $uv \in \K$, define $$T_u(v) := \{w \in V(\K): \vec{vw} \in E, uw \not\in \K\}$$ and $t_u(v) := |T_u(v)|$.  By construction, if $\vec{uv} \in E$, then $t_u(v) > 0$.  Also define $$T(v) := \{w \in V(\K): \vec{vw} \in E\}$$ and $t(v) := |T(v)|$.  Let $a'$ be the average value of $t(v)$ over all vertices $v$.  From $|U| \leq 3 \alpha_1 an$ we conclude that $a-a' \leq (3/2) \alpha_1 a$.  Furthermore, by construction $\frac{t_u(v)}{t(v)} \geq 1-3\alpha$ for all $\vec{uv} \in E$.

We now define a \textit{non-returning random walk} of length $k$, starting at a directed edge $e$, by a transition matrix $P$ with rows and columns indexed by $E$.  The entry $P_{e'e''}$ specifies the probability that in a random walk $\omega = (\omega_1,\ldots,\omega_p)$, if $\omega_i=e'$, then $\omega_{i+1}=e''$.  To construct $P$, every directed edge $\vec{vw}$ is given a positive weight $z_{vw}$.  If $T_u(v) = \{w_1, \ldots, w_s\}$, then for $1 \leq i \leq s$ set $$P_{\vec{uv},\vec{vw_i}} := \frac{z_{vw_i}}{z_{vw_1}+\ldots+z_{vw_s}}.$$ Otherwise, set $P_{e'e''}:=0$.

Let $x$ be the uniform probability distribution on $E$: $x_{\vec{uv}}=1/|E|$ for all directed edges $\vec{uv}$.  In Claim \ref{Claim1}, we show that the $z_{vw}$ can be chosen so that $x$ is a stable distribution under $P$, i.e. $xP=x$.  Furthermore, in the claim we show that there exists $\alpha_2$, which can be chosen arbitrarily small by choosing $\alpha$ sufficiently small, such that $1-\alpha_2 < z_{vw} < 1+\alpha_2$ for all $vw$.

For a given non-returning walk $\omega = (\vec{v_{-1}v_0},\vec{v_0v_1},\ldots,\vec{v_{k-1}v_k})$ with all edges in $E$, we denote by $p(\omega)$ the probably that $\omega$ is chosen among non-returning random walks of length $k+1$ starting at $\vec{v_{-1}v_0}$.  Since $P_{\vec{v_{i-1}v_i},\vec{v_iv_{i+1}}} \geq \frac{1-\alpha_2}{(1+\alpha_2)t_{v_{i-1}}(v_i)}$, $$p(\omega) \leq \left(\prod_{i=0}^{k-1} (1-\alpha_2)(1+\alpha_2)^{-1}t_{v_{i-1}}(v_i)\right)^{-1}.$$  There exists $\alpha_3 = 1-(1-\alpha_2)(1+\alpha_2)^{-1}(1-3\alpha)$, which can be chosen arbitrarily small by choosing $\alpha$ sufficiently small, such that $$p(\omega) \leq \left(\prod_{i=0}^{k-1} (1-\alpha_3)t(v_i)\right)^{-1}.$$

We repeat the calculations in \cite{MooreGraphs}.  Let $\Omega_{e,l}$ be the set of non-returning random walks of length $l+1$ starting at an edge $e$, and set $N_{e,l} := |\Omega_{e,l}|$.  Define $N_l := \sum_{e}x_eN_{e,l}$.  Using the AMGM inequality, $$N_l = \sum_{e}x_eN_{e,l} = \sum_e \sum_{\omega \in \Omega_{e,l}}x_e\frac{p(\omega)}{p(\omega)} \geq \prod_e \prod_{\omega} p(\omega)^{-x_ep(\omega)}.$$

If $n_{\vec{uv}}(\omega)$ is the number of instances of $\vec{uv}$ in a non-returning walk $\omega$, excluding the starting edge, then $$N_l \geq \prod_{\vec{uv}} (1-\alpha_3)T(v)^{\sum_{e} x_{e} \sum_{\omega \in \Omega_{e,l}}n_{\vec{uv}}(\omega)p(\omega)}.$$  The sum in the exponent is the expected number of visits, excluding the starting edge, to an edge $\vec{uv}$ if the starting edge is chosen randomly with the distribution $x$.  Since $x$ is stable under $P$, that quantity is $l/|E|$.

Then $$N_l \geq \prod_{\vec{uv}}((1-\alpha_3)T(v))^{l/|E|} \geq ((1-\alpha_3)a')^l$$ since $z^z$ is a log-convex function in $z$.  Hence there are, on average, at least $((1-\alpha_3)a')^l$ non-returning walks of length $l+1$ starting at a randomly chosen directed edge $\vec{uv}$.  Thus there exists a directed edge $\vec{uv}$ such that there are at least $((1-\alpha_3)a')^r$ non-returning paths of length $r+1$ starting at $\vec{uv}$, and an undirected edge $u'v'$ such that there are at least $2((1-\alpha_3)a')^r$ non-returning walks of length $r+1$ starting at either $\vec{u'v'}$ or $\vec{v'u'}$.  The theorem follows by Lemma \ref{TreeLemma}, by $a-a' \leq 3 \alpha_1 a$, and by taking $\alpha_1$ and $\alpha_3$ sufficiently small.
\endproof

\begin{claim}
\label{Claim1}
In the proof of Theorem \ref{FixedGirthStrong}, we can choose the $z_{vw}$ so that $xP=x$ and also so that $1-\alpha_2 < z_{vw} < 1+\alpha_2$ for all $\vec{vw} \in E$, where $\alpha_2$ can be chosen arbitrarily small by choosing $\alpha$ sufficiently small.
\end{claim}
\begin{proof}
Fix a vertex $v$, and let $z$ be a vector indexed by the directed edges of $E$ with starting vertex $v$.  Initially choose each $z_{vw} = 1$ and define $P = P(z)$ in terms of the $z_{vw}$ as above.  Let $P^*$ be the value of $P$ for the initial values of $z_{vw}$.  Also define $$D = D(z) := \sum_{w \in T(v)} |(xP(z))_{\vec{vw}}-x_{\vec{vw}}|.$$ Let $D^*$ be the initial value of $D$.  $D$ is a measure of how far $x$ is from being stable under $P$ around $v$.  We calculate 
\begin{equation}
\label{xpcalc}
(xP)_{\vec{vw}} = \sum_{u \in T_w(v)}x_{\vec{uv}}\frac{z_{vw}}{\sum_{w' \in T_u(v)}z_{vw'}}.
\end{equation}
Initially, $$\frac{1-3\alpha}{|E|} \leq \frac{t_w(v)}{|E|t(v)} \leq (xP^*)_{\vec{vw}} = \sum_{u \in T_w(v)}x_{\vec{uv}}\frac{1}{\sum_{w' \in T_u(v)}1} \leq \frac{(1-3\alpha)^{-1}}{|E|}.$$  Then there exists $\alpha_5 = (1-3\alpha)^{-1}-1$, which can be chosen arbitrarily small if $\alpha$ is chosen sufficiently small, such that $D^* < \alpha_5t(v)/|E|$.  Throughout the following construction, the value of $D$ always decreases, $(xP)_{\vec{vw}}$ never decreases unless to a value that is at least $1/|E|$, and $\sum_{w \in T(v)} z_{vw} = t(v)$.  Since
\begin{equation}
\label{ZVW}
\sum_{u \in T_w(v)}\frac{z_{vw}}{\sum_{w' \in T_u(v)}z_{vw'}} \geq \sum_{u \in T_w(v)}\frac{z_{vw}}{t(v)}  = t_w(v) \frac{z_{vw}}{t(v)} \geq (1-3\alpha)z_{vw},
\end{equation}
we have $z_{vw} \leq (1-3\alpha)^{-1}|E|(xP)_{\vec{vw}}$ and
\begin{equation}
\label{NotMuchError}
\sum_{w: z_{vw} \geq (1-3\alpha)^{-1}} (z_{vw}-(1-3\alpha)^{-1}) \leq (1-3\alpha)^{-1}D|E| \leq \alpha_6 t(v),
\end{equation}
where $\alpha_6 = (1-3\alpha)^{-1}\alpha_5$ can be chosen arbitrarily small by choosing $\alpha$ sufficiently small.  Furthermore,
\begin{equation}
\label{BonxP}
(xP)_{\vec{vw}} = \sum_{u \in T_w(v)}x_{\vec{uv}}\frac{z_{vw}}{\sum_{w' \in T_u(v)}z_{vw'}} = \sum_{u \in T_w(v)}x_{\vec{uv}}\frac{z_{vw}}{t(v) - \sum_{w' \not\in T_u(v)}z_{vw'}} \leq
\end{equation}
$$ \sum_{u \in T_w(v)}x_{\vec{uv}}\frac{z_{vw}}{t(v) - \alpha_6 t(v) - (1-3\alpha)^{-1}(t(v)-t_w(v))} \leq z_{vw}\frac{1+\alpha_7}{|E|},$$
where $\alpha_7$ can be chosen arbitrarily small by choosing $\alpha$ sufficiently small.  The second to last inequality follows from (\ref{NotMuchError}), and the fact that $t_w(v) \geq (1-3\alpha)t(v)$ allows us to choose $\alpha_7$ small.

Choose an edge $\vec{vw}$ so that $(xP)_{\vec{vw}}$ is maximal, say $1/|E|+b$.  Let $z'_{vw}$ be the larger of the following two values: (Case 1) $z_{vw}-D|E|/(2t(v))$, or (Case 2) the value necessary so that if we replace $z$ by $z'$ by replacing ${z_{vw}}$ by ${z'_{vw}}$, then $P' := P(z')$ satisfies $(xP')_{\vec{vw}} = 1/|E|$.  Update $z$ by replacing $z_{vw}$ with $z'_{vw}$ in $z$.  Since $(xP)_{\vec{vw}}$ is maximal, $b \geq {D}/(2t(v))$.  In Case 2, since $(xP)_{\vec{vw}}/(xP')_{\vec{vw}} = 1+b|E|$, it follows from (\ref{xpcalc}) that $z_{vw}/z'_{vw} \geq 1+b|E|$.  From $(xP)_{\vec{vw}} \geq 1/|E|$ and (\ref{BonxP}) we have that $z_{vw} \geq (1+\alpha_7)^{-1}$ and hence in Case 2 $z_{vw}-z'_{vw} \geq (1-\alpha_8)b|E|$, where $\alpha_8$ can be chosen arbitrarily small for $\alpha$ sufficiently small.  In either Case 1 or Case 2, $z_{vw}-z'_{vw} \geq (1-\alpha_8)D|E|/(2t(v))$.  For $w' \in T(v)$, we calculate
$$(xP')_{\vec{vw'}}-(xP)_{\vec{vw'}} = \frac{z_{vw'}}{|E|}\left(\sum_{u \in T_{w'}(v)}\frac{1}{\sum_{\tilde{w} \in T_u(v)}z'_{v\tilde{w}}} - \sum_{u \in T_{w'}(v)}\frac{1}{\sum_{\tilde{w} \in T_u(v)}z_{v\tilde{w}}}\right)= $$
$$\frac{z_{vw'}}{|E|}\sum_{u \in T_{w'}(v) \cap T_w(v)}\left( \frac{1}{\sum_{\tilde{w} \in T_u(v)}z'_{v\tilde{w}}} - \frac{1}{\sum_{\tilde{w} \in T_u(v)}z_{v\tilde{w}}} \right) \geq $$
$$\frac{z_{vw'}}{|E|}(1-6\alpha)t(v)\left(\frac{1}{t(v)-(1-\alpha_8)D|E|/(2t(v))}-\frac{1}{t(v)} \right). $$
The second equality follows from the fact that the two fractions are equal if $u \not\in t_w(v)$.  Since $(xP)_{\vec{vw'}} \geq (1-3\alpha)/|E|$ (by the fact that $(xP^*)_{\vec{vw'}} \geq (1-3\alpha)/|E|$ and $(xP)_{\vec{vw'}}$ does not decrease to below $(1-3\alpha)/|E|$) and $z_{vw'} \geq |E|(xP)_{\vec{vw'}}(1+\alpha_7)^{-1} \geq (1-3\alpha)(1+\alpha_7)^{-1}$, we have
$$(xP')_{\vec{vw'}}-(xP)_{\vec{vw'}} \geq \frac{(1-\alpha_9)D}{2t(v)^2},$$
where $\alpha_9$ can be chosen arbitrary small for $\alpha$ sufficiently small.  There exists a $w'$ with $(xP)_{\vec{vw'}} \leq x_{\vec{vw'}} - {D}/(2t(v))$, and so $D$ decreases by at least $(1-\alpha_9)\frac{D}{2t(v)^2}$ under the replacement of $z$ with $z'$  This follows from the fact that all the $(xP)_{\vec{vw'}}$ increase except for $(xP)_{\vec{vw}}$, which decreases to a value at least $x_{vw}$.  After replacing $z_{vw}$ by $z'_{vw}$, rescale the $z_{vw'}$ so that $\sum_{w' \in T(v)} z_{vw'} = t(v)$.

By repeating this process, $D \rightarrow 0$ exponentially.  It follows that the change in all the $z_{vw}$ in one iteration of the above process also decreases exponentially, and therefore the $z_{vw}$ converge to values for which $xP=x$.  Furthermore, each $1-\alpha_2 < z_{vw} < 1+\alpha_2$ for each $z_{vw}$ by (\ref{ZVW}) and (\ref{BonxP}), where $\alpha_2$ can be chosen arbitrarily small by choosing $\alpha$ sufficiently small.
\end{proof}

Our proof depends on the assumption that the average degree of a vertex in $\K$ is sufficiently large relative to $d$, and that is the reason for the $r/\log (n/d) < \delta$ hypothesis.

In some cases, an improvement to Theorem \ref{FixedGirthStrong} is possible.  For example, consider the case that $d=3$ and $r=3$.  Let $u,v \in V(\K)$, and consider paths from $u$ to $v$ of the form $(u,u',v',v)$.  By the argument of Lemma \ref{TreeLemma}, $v'$ could take on at most two values, say $v_1$ and $v_2$, and that can happen only if there is an edge $v_1v_2$.  Similarly, by considering paths from $v$ to $u$, $u'$ can only take on at most two values, say $u_1$ and $u_2$, and that can happen only if there is an edge $u_1u_2$.  Since $\K$ is flag, if all four edges $u_1v_1, u_1v_2, u_2v_1, u_2v_2$ are in $\K$, then $\K$ has a $3$-face $u_1u_2v_1v_2$, a contradiction to $d=3$.  Hence there are at most $3$ paths of length $3$ from $u$ to $v$.  Following the proof of Theorem \ref{FixedGirthStrong}, we conclude that for all $\epsilon > 0$, $f_1(\K) \leq (2^{-1}+\epsilon)3^{1/3}n^{4/3}$ for sufficiently large $n$.

\section{One-girth and higher face numbers}
\label{OneGirthLargeI}

Next we prove an analogue to Theorem \ref{FixedGirthStrong} for higher face numbers.  The following result bounds higher face numbers when the $1$-girth is given.

\begin{theorem}
\label{HighDim1Girth}
Let $\K$ be a $(d-1)$-dimensional simplicial complex with $n$ vertices, and suppose $\g_1(\K) > 2r$, $r \geq 2$.  Then for some constant $C_{r,i}$ that depends only on $r$ and $i$, $$f_i(\K) \leq C_{r,i}d^{i-1/r-1/r^2- \ldots - 1/r^i}n^{1+1/r+1/r^2+ \ldots + 1/r^i}.$$
\end{theorem}
\begin{proof}
We use induction on $i$, with the case that $i=0$ trivial.  By Part 2 of Lemma \ref{GirthProperties} and the inductive hypothesis, for a vertex $v \in V(\K)$, $$f_{i-1}(\lk_\K(v)) \leq C_{r,i-1}d^{i-1-1/r-1/r^2- \ldots - 1/r^{i-1}}(\deg v)^{1+1/r+1/r^2+ \ldots + 1/r^{i-1}}.$$  Then, since $(i+1)f_i(\K) = \sum_{v \in V(\K)} f_{i-1}(\lk_\K(v))$, $$f_i(\K) \leq \frac{1}{i+1}C_{r,i-1} \sum_{v \in V(\K)} d^{i-1-1/r-1/r^2- \ldots - 1/r^{i-1}}(\deg v)^{1+1/r+1/r^2+ \ldots + 1/r^{i-1}}.$$  The theorem follows by Lemma \ref{PowerDegree} with $p = 1+1/r+1/r^2+ \ldots + 1/r^{i-1}$.
\end{proof}

\begin{lemma}
\label{PowerDegree}
Let $\K$ be a $(d-1)$-dimensional simplicial complex with $n$ vertices and $\g_1(\K) > 2r$.  Then for a fixed $p<\frac{r}{r-1}$, there exists a constant $C$, which depends only on $r$ and $p$, such that $$\sum_{v \in V(\K)} (\deg v)^p \leq C d^{p-p/r}n^{1+p/r}.$$
\end{lemma}
\begin{proof}
For a fixed $\alpha > 0$, we may assume that $d < \alpha n$ by choosing $C$ sufficiently large.  First we show that there exists a constant $C_1$, which depends only on $r$, such that for each $1 \leq R \leq d^{-1+1/r}n^{1-1/r}$, $\K$ does not contain more than $C_1nR^{\frac{-r}{r-1}}$ vertices with degree at least $Rd^{1-1/r}n^{1/r}$.  Suppose by way of contradiction that, for $C_1$ arbitrarily large in terms of $r$, there are $C_1nR^{\frac{-r}{r-1}}$ vertices with degree at least $Rd^{1-1/r}n^{1/r}$ and call this set of vertices $X$.

We consider non-returning walks on $\vec{\Skel}_1(\K)$.  Construct a directed bipartite graph $G(V,E)$ with $V(G) = X \sqcup Y$, $Y := V(\K)$.  We say that there are directed edges $\vec{xy}$ and $\vec{yx}$ in $G$ joining $x \in X$ and $y \in Y$ if $xy$ is an edge in $\K$.  There is a bijection between walks $(\vec{v_{-1}v_0}, \ldots, \vec{v_{k-1}v_k})$ in $\K$ with $v_0,v_2, \ldots \in X$, and walks in $G$ with the endpoint of the initial directed edge in $X$.  The average out-degree in $G$ of the vertices of $Y$ is at least $C_1d^{1-1/r}n^{1/r}R^{\frac{-1}{r-1}}$.

We define a set of edges in $G$ that are allowed in our walks.  First define $U'$ to be the set of all edges $\vec{uv}$ or $\vec{vu}$ with $u \in X, v \in Y$ that satisfy at least one of the following three conditions:

1) $f_0(\lk_\K(uv)) \geq (1/6) |V(\lk_\K(u))|$,

2) $|V(\lk_\K(uv)) \cap X| \geq (1/6) |V(\lk_\K(v)) \cap X|$,

3) $v$ has out-degree less than $\frac{C_1}{4} d^{1-1/r}n^{1/r}R^{\frac{-1}{r-1}}$. \newline
Suppose that $C_1$ is chosen sufficiently large and $\alpha$ sufficiently small (both independently of $d$ or $n$).  For a fixed $u \in X$, we have that there are at most $(1/40) |V(\lk_\K(u))|$ edges incident to $u$ in the first category by applying Lemma \ref{HighDegLemma} to $\lk_\K(u)$, and so there are at most $(1/40)|E|$ edges in the first category.  Likewise, consider a fixed $v \in Y$.  There are at most $(1/40) |V(\lk_\K(v)) \cap X|$ edges incident to $v$ in the second category but not the third by applying Lemma \ref{HighDegLemma} to $\lk_\K(v) \cap \K[X]$, and so there are at most $(1/40)|E|$ edges in the second category but not the third.  Also, there are fewer than $(1/4)|E|$ directed edges in the third category.  Then $|U'| < \frac{1.2}{4} |E|$.

Now set $U := U'$.  If there is a vertex $v$ in either side of $G$ such that $(2/3) \deg v$ of the directed edges incident to $v$ are in $U$, then add all edges incident to $v$ to $U$.  Repeat until no more edges can be added in this manner.  It follows from the same argument as in the proof of Theorem \ref{FixedGirthStrong} that $|U| < \frac{3.6}{4} |E|$.  Furthermore, for every $x \in X$, there are either $0$ or at least $Rd^{1-1/r}n^{1/r}/3$ directed edges in $E-U$ that start at $x$, and for every $y \in Y$ there are either $0$ or at least $\frac{C_1}{12} d^{1-1/r}n^{1/r}R^{\frac{-1}{r-1}}$ directed edges in $E-U$ that start at $y$.  For a vertex $v \in V(G)$, let $T(v)$ denote the number of directed edges in $E-U$ that start at $v$.  By construction, this is the same as the number of directed edges in $E-U$ that end at $v$.

Consider walks of length $r$ in $G$, starting with an edge $v_{-1}v_0$ with $v_0 \in X$ and using edges in $E-U$, such that the corresponding path in $\K$ is a non-returning walk.  For every edge $\vec{uv}$ or $\vec{vu}$ in such a path with $u \in X, v \in Y$ we have $f_0(\lk_\K(uv)) < (1/6)f_0(\lk_\K(u)) \leq (1/2) T(u)$ and $|V(\lk_\K(uv)) \cap X| < (1/6)|V(\lk_\K(v)) \cap X| \leq (1/2) T(v)$.  Then for some value $C_2$ that can be chosen arbitrary large by choosing $C_1$ sufficiently large, there are at least $$\left(\frac{1}{6}Rd^{1-1/r}n^{1/r}\right)^{\lceil \frac{r}{2} \rceil}\left(\frac{C_1}{24}R^{\frac{-1}{r-1}}d^{1-1/r}n^{1/r}\right)^{\lfloor \frac{r}{2} \rfloor} = C_2nR^{\lceil \frac{r}{2} \rceil-\frac{\lfloor \frac{r}{2} \rfloor}{r-1}}d^{r-1} \geq C_2nd^{r-1}$$ such paths.  It follows from Lemma \ref{TreeLemma} that $\K$ has more than $n$ vertices for $C_2 > 1$, a contradiction.

Let $W_R$ be the set of vertices with degree between $Rd^{1-1/r}n^{1/r}$ and $2Rd^{1-1/r}n^{1/r}$.  We have shown that $|W| \leq C_1nR^{\frac{-r}{r-1}}$ for some $C_1$ that depends only on $r$.  Then $$\sum_{v \in W_R} (\deg v)^p \leq C_3d^{p-p/r}n^{1+p/r}R^{p-\frac{r}{r-1}}$$ for some constant $C_3$.  By adding over all $R=1,2,4,\ldots,2^{\lfloor \log_2 d^{-1+1/r}n^{1-1/r} \rfloor}$ and by $p < \frac{r}{r-1}$, it follows that $\sum_{v \in V(\K)} (\deg v)^p \leq Cd^{p-p/r}n^{1+p/r}$ for some constant $C$ that depends only on $p$ and $r$.
\end{proof}

The following conjecture, which is reasonable in light of Theorem \ref{FixedGirthStrong}, is a possible strengthening of Theorem \ref{HighDim1Girth}.

\begin{conjecture}
Let $\K$ be a $(d-1)$-dimensional simplicial complex with $n$ vertices, and suppose that $\g_1(\K) > 2r$, $r \geq 2$.  For every $\epsilon > 0$, there exists $\delta > 0$ such that if $r/\log(n/d) < \delta$, then $$f_i(\K) \leq \left(\frac{1}{(i+1)!}+\epsilon\right)(d-1)^{1-{1}/{r^i}}(d-2)^{1-{1}/{r^{i-1}}}\ldots (d-i)^{1-{1}/{r}}n^{1+{1}/{r}+\ldots+{1}/{r^i}}.$$
\end{conjecture}

While the theorems in this section and in Section \ref{OneGirth} only apply when $\g_1(\K) \geq 5$, the case that $\g_1(\K) = 4$ is fully addressed by earlier results.  Given that $\g_1(\K) \geq 4$ (i.e. $\K$ is flag), and that $\K$ has dimension $d-1$ and $n$ vertices, then all face numbers are simultaneously maximized by the following construction.  Partition $V(\K)$ into $d$ sets $V_1, \ldots, V_d$ as evenly as possible, and let all vertex subsets that consist of at most one element from each of the $V_i$ be faces of $\K$.  This result is proven in \cite{Flag}.

\section{Higher girths}
\label{HighGirth}
Next we turn our attention to bounds on face numbers that arise from higher girth assumptions.  In this section we conjecture an upper bound on $f_{i-1}(\K)$ when the dimension, number of vertices, and $(p-1)$-girth of $\K$ are given.

We define the exponents used in the following conjecture recursively.  Define
\begin{eqnarray*}
& a_{2,r,i} & := \frac{r^i-1}{r^i-r^{i-1}} = 1 + 1/r + \ldots + 1/r^{i-1}, \\
& a_{p,r,p-1} & := p-1, \\
& a_{p,r,i} & := \frac{1}{2} a_{p-1,r,i-1} + \frac{1}{2} a_{p,r,i-1} + 1, \quad p \geq 3, i \geq p.
\end{eqnarray*}
 We note some properties of the $a$ values.  For all $p \geq 3, r \geq 2, i \geq p$, we have $a_{p,r,i} > a_{p-1,r,i}, a_{p,r,i} > a_{p,r,i-1}$, and $a_{p,r,i} > a_{p,r+1,i}$.  Also, $a_{p,r,i} < 2p-3+\frac{1}{r-1}$, and $\lim_{i \rightarrow \infty} a_{p,r,i} = 2p-3+\frac{1}{r-1}$.  Each of these properties can be checked by using using induction on $p$ and $i$.

\begin{conjecture}
\label{GenConjecture}
Let $\K$ be a $(d-1)$-dimensional simplicial complex with $n$ vertices, and suppose that $\g_{p-1}(\K) > 2p+2r-4$, for some $p,r \geq 2$.  Then for a constant $C_{p,r,i}$ that depends only on $p,r,i$, $$f_{i-1}(\K) \leq C_{p,r,i}d^{i-a_{p,r,i}}n^{a_{p,r,i}}.$$
\end{conjecture}

We do not have a general proof of Conjecture \ref{GenConjecture}, so we prove the conjecture in several special cases.  Our next theorem verifies the conjecture for flag complexes.  We use the notation $\avg_{t \in T} f(t)$ to denote the average value of a real-valued function $f(t)$ as $t$ ranges over all elements of a finite set $T$.

\begin{theorem}
\label{generalflag}
Let $\K$ be a flag $(d-1)$-dimensional simplicial complex with $n$ vertices, and suppose that $\g_{p-1}(\K) > 2p+2r-4$, for some $p,r \geq 2$.  Then for a constant $C_{p,r,i}$ that depends only on $p,r,i$, $$f_{i-1}(\K) \leq C_{p,r,i}d^{i-a_{p,r,i}}n^{a_{p,r,i}}.$$
\end{theorem}
\begin{proof}
We prove the result by induction on $p$ and $i$.  The case that $p=2$ is a restatement of Theorem \ref{HighDim1Girth}, and the case that $i=p-1$ is trivial.  Now suppose that $p \geq 3$ and $i \geq p$, and let $f_{i-1}$ be given.

By the inductive hypothesis, $f_{i-2} \leq C_{p,r,i-1}d^{i-1-a_{p,r,i-1}}n^{a_{p,r,i-1}}$.  For a fixed value $R$ that is independent of $d$ and $n$, we may assume that $$f_{i-1} > RdC_{p,r,i-1}d^{i-1-a_{p,r,i-1}}n^{a_{p,r,i-1}};$$ otherwise, then the theorem follows by $a_{p,r,i} > a_{p,r,i-1}$ and by choosing $C_{p,r,i}$ sufficiently large.  The $(i-2)$-faces of $\K$ are contained in at least $\frac{if_{i-1}}{C_{p,r,i-1}d^{i-1-a_{p,r,i-1}}n^{a_{p,r,i-1}}}$ $(i-1)$-faces on average, or $$\avg_{F \in \K, |F|=i-1}f_0(\lk_\K(F)) = \frac{if_{i-1}(\K)}{f_{i-2}(\K)} \geq \frac{if_{i-1}}{C_{p,r,i-1}d^{i-1-a_{p,r,i-1}}n^{a_{p,r,i-1}}} > Rdi.$$  For an $(i-1)$-face $F$, let $q(F) := f_0(\lk_\K(F))$ if $f_0(\lk_\K(F)) > Rdi/2$ and $0$ otherwise.  It follows that $$\avg_{F \in \K, |F|=i-1}q(F) \geq \frac{if_{i-1}(\K)}{2f_{i-2}(\K)}.$$

For a simplicial complex $\Delta$, let $s(\Delta)$ denote the number of pairs of vertices that are \textit{not} joined by an edge; then $s(\Delta) = {f_0(\Delta) \choose 2} - f_1(\Delta)$.  Since $\K$ is flag, we conclude that for some $\epsilon$ that depends only on $p$, $i$, and $r$, and for all $(i-1)$-faces $F$, $s(\lk_\K(F)) \geq \epsilon q(F)^2$.  This follows from the first part of Lemma \ref{HighDegLemma} if $f_0(\lk_\K(F)) > Rdi/2$ and is trivial otherwise.  From the fact that for any set $T$ of real numbers, $\avg_{t \in T}(t^2) \geq (\avg_{t \in T})^2$, we conclude that there exists a constant $C_1$ independent of $d$ or $n$ such that $$\avg_{F \in \K, |F|=i-1}s(\lk_\K(F)) \geq \frac{C_1f_{i-1}(\K)^2}{f_{i-2}(\K)^2}.$$

For some constant $C_2$ independent of $d$ or $n$, there are at least $$\frac{C_1f_{i-1}(\K)^2}{f_{i-2}(\K)} \geq \frac{C_2f_{i-1}^2}{d^{i-1-a_{p,r,i-1}}n^{a_{p,r,i-1}}}$$ sets of the form $\{F,v,v'\}$, where $F$ is an $(i-2)$-face and $v,v'$ are vertices in $\lk_\K(F)$ that are not joined by an edge.  By Part 1 of Lemma \ref{FlagProperties}, if there is no edge $vv'$ between two vertices $v$ and $v'$ in $\lk_\K(F)$ for a face $F$, then there is no edge $vv'$ in $\K$.  There are at most ${n \choose 2}$ pairs of vertices $v$ and $v'$ that are not joined by an edge; call that set of pairs $S$.  For all $(v,v') \in S$, let $k_{i-2}(v,v')$ be the number of $(i-2)$-faces whose links contain both $v$ and $v'$.  Then $$\avg_{(v,v')\in S}k_{i-2}(v,v') \geq \frac{C_3f_{i-1}^2}{d^{i-1-a_{p,r,i-1}}n^{2+a_{p,r,i-1}}}$$ for some constant $C_3$ that depends only on $C_2$.

Choose $(v,v') \in S$ so that $$k_{i-2}(v,v') \geq \frac{C_3f_{i-1}^2}{d^{i-1-a_{p,r,i-1}}n^{2+a_{p,r,i-1}}}.$$  Let $W$ be the set of vertices incident to both $v$ and $v'$, so that $$f_{i-2}(\K[W]) \geq \frac{C_3f_{i-1}^2}{d^{i-1-a_{p,r,i-1}}n^{2+a_{p,r,i-1}}}.$$ By the inductive hypothesis, $\g_{p-2}(\K[W]) \leq 2p+2r-6$ if $$\frac{C_3f_{i-1}^2}{d^{i-1-a_{p,r,i-1}}n^{2+a_{p,r,i-1}}} \geq C_{p-1,n,i-1}d^{i-1-a_{p-1,r,i-1}}n^{a_{p-1,n,i-1}}.$$ By Part 3 of Lemma \ref{FlagProperties}, there exists $W' \subset W$ so that $|W| \leq 2p-2r-6$ and $\tilde{H}_{p-2}(\K[W'];\field) \neq 0$. Then $\K[W',v,v']$ is the suspension of $\K[W']$, which implies that $\tilde{H}_{p-1}(\K[W,v,v']; \field) \neq 0$ and $\g_{p-1}(\K[W]) \leq 2p+2r-4$, a contradiction.  Therefore, $$\frac{C_3f_{i-1}^2}{d^{i-1-a_{p,r,i-1}}n^{2+a_{p,r,i-1}}} < C_{p-1,n,i-1}d^{i-1-a_{p-1,r,i-1}}n^{a_{p-1,n,i-1}}.$$ We conlcude that, for a constant $C_{p,r,i}$, 
\begin{eqnarray*}
& f_{i-1} & < C_{p,r,i}d^{(i-1)/2-a_{p,r,i-1}/2+(i-1)/2-a_{p-1,r,i-1}/2}n^{a_{p,r,i-1}/2+a_{p-1,r,i-1}/2+1} \\
& & =C_{p,r,i}d^{i-a_{p,r,i}}n^{a_{p,r,d,i}}.
\end{eqnarray*}
\end{proof}

The next special case verifies Conjecture \ref{GenConjecture} in the case that both $r=2$ and $i=p$.  First we need some technical lemmas.

\begin{lemma}
\label{IncreasingGirths}
Let $\K$ be a simplicial complex, and suppose that $\g_{p}(\K) < \infty$.  Then $\g_{p-k}(\K) \leq \g_p(\K)-k$ for all $0 \leq k \leq p$.
\end{lemma}
\begin{proof}
It suffices to prove the lemma for $k=1$.  Let $F \in \K$ and $W \subset V(\K)$ so that $|W| = \g_p(\K)$ and $\tilde{H}_p(\lk_\K(F)[W]);\field) \neq 0$.  Choose $v \in W$.  By definition of $\g_p$, $\tilde{H}_p(\lk_\K(F)[W-\{v\}];\field) = 0$.  It follows from the Mayer-Vietoris sequence that $\tilde{H}_{p-1}(\lk_\K(F \cup \{v\})[W-\{v\}]);\field) \neq 0$, proving the result.
\end{proof}

\begin{lemma}
\label{CPLemma}
Let $d$ and $p$ be fixed, and let $\Delta$ be a $(d-1)$-dimensional simplicial complex with $V(\Delta) = T \sqcup W, |T| = n$.  Suppose that for all $v \in T$, $\tilde{H}_{p-1}(\lk_\Delta(v)[W]; \field) \neq 0$.  Then for some $i \geq 2$ and for some constant $C$ that depends on $i,d,|W|,p$, there exist at least $Cn^i$ subsets $W' \subset T$ satisfying $|W'|=i$ and $\tilde{H}_{p+i-2}(\Delta[W' \cup W]; \field) \neq 0$ if $n$ is sufficiently large.
\end{lemma}
\begin{proof}
We prove the result by induction on $d$.  Let $\Lambda$ be the simplicial complex that is the value of $\lk_\Delta(v)[W]$ for the largest number of vertices $v$.  We restrict attention to the subcomplex $\Delta' = \Delta[T',W]$, where $T'$ is the set of vertices $v$ satisfying $\lk_\Delta(v)[W] = \Lambda$.  Then for a (possibly very small) constant $C'$ that depends only on $|W|$, $|T'| \geq C'|T|$.

If there are ${|T'| \choose 2}/2 \geq {C'n \choose 2}/2$ pairs of vertices $\{u,v\} \subset T'$ such that $\tilde{H}_p(\Delta[W \cup \{u,v\}];\field) \neq 0$, then the result holds with $i=2$.  Otherwise, by the Mayer-Vietoris sequence, for at least ${|T'| \choose 2}/2$ pairs of vertices $\{u,v\} \subset T'$, there is an edge $uv$ and $\tilde{H}_{p-1}(\lk_\Delta(uv)[W];\field) \neq 0$.  Let $T'_v := \{u: \tilde{H}_{p-1}(\lk_\Delta(uv);\field) \neq 0\}$.  Then $\avg_{v \in T'}|T'_v| \geq (|T'|-1)/2 \geq (C'n-1)/2$.  Since $|T'_v| \leq |T'|$ for all $v$, there exists a set $V \subset T$ so that $|V| \geq |T'|/5$ and for all $v \in V$, $|T'_v| \geq |T'|/5$.

For $v \in V$, $\lk_\Delta(v)[T'_v,W]$ satisfies the conditions of the lemma and so, by the inductive hypothesis, for some $i_v$ and constants $C_1$ and $C'_1$ that depend only on $i_v,d,|W|,p$, there exist at least $C_1(C'n/5)^{i_v} = C'_1n^{i_v}$ subsets $W' \subset T'_v$ so that $W'=i_v$ and $\tilde{H}_{p+i_v-2}(\lk_\Delta(v)[W' \cup W];\field) \neq 0$.  If for some $v \in V$, $\tilde{H}_{p+i_v-2}(\Delta[W' \cup W]; \field) \neq 0$ for $(C_1'/2)n^{i_v}$ such values of $W'$, the lemma is proven by taking $C = C_1'/2$.  Otherwise, by the Mayer-Vietoris sequence, there exist $(C_1'/2)n^{i_v}$ values of $W'$ so that $|W'| = i_v$ and $\tilde{H}_{p+i_v-1}(\Delta[W' \cup \{v\} \cup W];\field) \neq 0$.  Choose $i$ so that $i$ is the value of $i_v$ for the largest number of vertices $v \in V$.  Since $i_v \leq d$ for all $v$, there are at least $\frac{C'}{5d}n$ vertices $v$ so that there exist $(C_1'/2)n^{i}$ values of $W'$ so that $|W'|=i$ and $\tilde{H}_{p+i-1}(\Delta[W' \cup \{v\} \cup W];\field) \neq 0$.  We conclude that there exist at least $\frac{C'C_1'}{10d(i+1)}n^{i+1}$ values of $W'$ so that $|W'| = i+1$ and $\tilde{H}_{p+i-1}(\Delta[W' \cup W];\field) \neq 0$.  This proves the lemma.
\end{proof}

\begin{theorem}
\label{Crosspolytopes}
Let $\K$ be a $(d-1)$-dimensional simplicial complex with $n$ vertices, and suppose that $\g_{p-1}(\K) > 2p$ for some $p \geq 2$.  Then for a constant $C_{p,d}$ that depends only on $p$ and $d$, $$f_{p-1}(\K) \leq C_{p,d}n^{p-\frac{1}{2^{p-1}}} = C_{p,d}n^{a_{p,2,p}}.$$
\end{theorem}
\begin{proof}
We use the notion of an $(s,q)$-\textit{open cycle}, which is a subset $W=\{w_1,\ldots,w_{q}\} \subset V(\K)$ so that $\tilde{H}_{s-1}(\K[W];\field) \neq 0$.  Let $\CP_{s,q}(\K)$ be the set of $(s,q)$-open cycles in $\K$, and let $\cp_{s,q}(\K) := |\CP_{s,q}(\K)|$.

Adding a face $F$ with $|F|<p$ to $\K$ does not affect $\g_{p-1}(\K)$, and so we assume without loss of generality that all sets of cardinality less than $p$ are faces in $\K$.

Suppose that $R$ is large, independently of $n$, and that $f_{p-1} > Rn^{p-\frac{1}{2^{p-1}}}$.  Then we show that there exists a face $F$ of $\K$ (possibly the empty face) such that $\lk_\K(F)$ contains an $(s,q)$-open cycle for some $s \geq p$ and $q \leq s+p$.   The theorem then follows by Lemma \ref{IncreasingGirths}.

Since $$\avg_{F \in \K, |F|=p-1}f_0(\lk_\K(F)) \geq C Rn^{1-\frac{1}{2^{p-1}}}$$ for some constant $C$ independent of $n$, by Lemma \ref{CPLemma} (with $W = \emptyset$) there exists for each $F$ $2 \leq q_1(F) \geq d+1$ and constant $C'_1$ independent of $n$ so that if $f_0(\lk_\K(F))$ is sufficiently large, 
\begin{equation*}
\cp_{q_1-1,q_1}(\lk_\K(F)) \geq C'_1 f_0(\lk_\K(F))^{q_1(F)}.
\end{equation*}
By considering the value of $q_1$ so that $\sum_{|F|=p-1, q_1(F)=q_1}f_0(\lk_\K(F))$ is maximal, and the fact that $\avg_{t \in T}(t^q) \leq (\avg_{t \in T}(t))^q$ for all $q \geq 1$, we have that
\begin{equation*}
\avg_{F \in \K, |F|=p-1}\cp_{q_1-1,q_1}(\lk_\K(F)) \geq C_1Rn^{q_1-\frac{q_1}{2^{p-1}}}.
\end{equation*}
If $q_1 \geq p+1$, then the result is proven, so suppose that $q_1 \leq p$.

Now suppose that we have found $j$, $q_1, \ldots, q_{j-1}$ such that $Q := q_1+\ldots+q_{j-1} \leq p+j-2$, and constant $C_{j-1}$ independent of $n$ such that 
\begin{equation}
\label{CPEquation}
\avg_{F \in \K, |F|=p-j+1}\cp_{Q-j+1,Q}(\lk_\K(F)) \geq C_{j-1}Rn^{Q-\frac{q_1 \ldots q_{j-1}}{2^{p-1}}}.
\end{equation}
Consider $F$ with $|F|=p-j$.  For each $W \subset V(\K), |W|=Q$, let $U_{W,F}$ be the set of vertices $v$ such that $\tilde{H}_{Q-j}(\lk_\K(F \cup \{v\})[W]; \field) \neq 0$.  From (\ref{CPEquation}) and the fact that all $(p-j+1)$-subsets of $\K$ are faces, we have that $$\avg_{F \in \K, |F|=p-j}\sum_{v \in V(\K)-F}\cp_{Q-j+1,Q}(\lk_\K(F \cup \{v\})) \geq C_{j-1}Rn^{Q-\frac{q_1 \ldots q_{j-1}}{2^{p-1}}}(n-p+j).$$  Since there are fewer than $n^{Q}\frac{n-p+j}{n}$ subsets $W$ of size $Q$ of $V(\K)$, it follows that $$\avg_{F \in \K, |F|=p-j, W \subset V(\K), |W|=Q}|U_{W,F}| \geq C_{j-1}Rn^{1-\frac{q_1 \ldots q_{j-1}}{2^{p-1}}},$$ and since $1-q_1 \ldots q_{j-1}2^{-p+1} \geq 0$, $$\avg_{F \in \K, |F|=p-j}\cp_{Q+q_j-j,Q+q_j}(\lk_\K(F)) \geq C_jRn^{Q+q_j-\frac{q_1 \ldots q_j}{2^{p-1}}}$$ for some $q_j \geq 2$ and $C_j$ independent of $n$, again by Lemma \ref{CPLemma} and the above reasoning.  The theorem follows if $Q+q_j \geq p+j$.  Otherwise, repeat this argument until such $j$ is found.
\end{proof}

Our last special case verifies Conjecture \ref{GenConjecture} in the case that $i=p=d$.

\begin{theorem}
Let $\K$ be a $(p-1)$-dimensional simplicial complex with $n$ vertices, and suppose that $\g_{p-1}(\K) > 2p+2r-4$, for some $p \geq 2$ and $r \geq 2$.  Then for a constant $C_{p,r}$ that depends only on $p$ and $r$, $$f_{p-1}(\K) \leq C_{p,r}n^{p-\frac{4r-4}{r2^p}} = C_{p,r}n^{a_{p,r,p}}.$$
\end{theorem}
\begin{proof}
We prove by induction on $p$ that if $\K$ has dimension $p-1$ and satisfies $f_{p-1}(\K) > C_{p,r}n^{p-\frac{4r-4}{r2^p}}$, then there exists $W \subset V(\K)$ such that $|W| \leq 2p+2r-4$ and $\K[W]$ has a non-trivial $(p-1)$-cycle in homology.  The theorem then follows since that cycle cannot be a boundary in a $(p-1)$-dimensional simplicial complex.  For $p=2$, the claim follows from Theorem \ref{IrregularMoore}.

Suppose that $f_{p-1}(\K) > C_{p,r}n^{p-\frac{4r-4}{r2^p}}$ for sufficiently large $C_{p,r}$ independent of $n$.  Then since every $(p-1)$-face contains $p$ faces with $p-2$ vertices each, $$\sum_{F \in \K, |F|=p-1}f_0(\lk_\K(F)) > pC_{p,r}n^{p-\frac{4r-4}{r2^p}} \Rightarrow$$ $$\avg_{F \in \K, |F|=p-1}f_0(\lk_\K(F)) > pC_{p,r}n^{p-\frac{4r-4}{r2^p}}f_{p-2}^{-1}.$$ Take $s(\Delta)$ to be the number of pairs of vertices in $\Delta$ that are not joined by an edge, as in the proof of Theorem \ref{generalflag}.  Since $\K$ has dimension $p-1$, the link of a $(p-2)$ face contains no edges, and by $\avg_{t \in T}(t^2) \geq (\avg_{t \in T})^2$, we have $$\avg_{F \in \K, |F|=p-1}s(\lk_\K(F)) > Cn^{2p-\frac{4r-4}{r2^{p-1}}}f_{p-2}^{-2} \Rightarrow $$ $$\sum_{F \in \K, |F|=p-1}s(\lk_\K(F)) > Cn^{2p-\frac{4r-4}{r2^{p-1}}}f_{p-2}^{-1}$$ for a value $C$ that can be chosen arbitrarily large by choosing $C_{p,r}$ sufficiently large. Since there are ${n \choose 2}$ pairs of vertices in $V(\K)$ and $f_{p-2}(\K) \leq {n \choose p-1}$, there exist vertices $u$ and $v$ such that there exists a set $P$ of $Cn^{p-1-\frac{4r-4}{r2^{p-1}}}$ $(p-2)$-faces whose links contain $u$ and $v$.  Let $\K'$ be the simplicial complex with maximal faces given by $P$.  By the inductive hypothesis, for sufficiently large $C$ there exists $W' \subset V(\K')$ such that $|W'| \leq 2p+2r-6$ and $\K'[W']$ has a non-trivial $(p-2)$-cycle in homology.  Then $\K[W',\{u,v\}]$ contains the suspension of $\K'[W]$ and therefore has a non-trivial $(p-1)$-cycle in homology, and thus $\g_{p-1}(\K) \leq 2p+2r-4$.
\end{proof}

\begin{problem}
Improve the exponent in the bound of Conjecture \ref{GenConjecture} (or a special case of it), or give an example to show that such improvement is impossible.  Also, what specific values of $C_{p,r,i}$ can be given?
\end{problem}

\begin{problem}
What bounds on face numbers are possible if there are several girth hypotheses, say if $\g_1$ and $\g_2$ are given?
\end{problem}

\section{Existence of complexes with high girth and many faces}
\label{Existence}

Although we do not have general examples to prove that the bounds in the above theorems are tight, we have some examples of simplicial complexes satisfying the hypotheses of Conjecture \ref{GenConjecture} and having $f_{p-1}$ large.

\begin{theorem}
For each $p \geq 2$ there exists a constant $C_p$ such that for arbitrarily large $n$, there exists a $(p-1)$-dimensional simplicial complex $\K$ with $n$ vertices satisfying $\g_{p-1}(\K) > 2p$ and $f_{p-1}(\K) \geq C_p n^{p-\frac{p}{2^p-1}}$.
\end{theorem}
\begin{proof}
Without loss of generality, we assume that $n/p$ is an integer.  Let $V$ be a set of $n$ vertices, partitioned into sets $V_1, \ldots, V_p$, each of size $n/p$.  Construct a random simplicial complex $\K'$ as follows.  All faces with cardinality at most $p-1$ and at most one element of each of the $V_i$ are faces in $\K'$, and every cardinality $p$ set with exactly one element from each $V_i$ is added to $\K'$ independently with probability $a$ for some $0 < a < 1$.  A $(p-1)$-dimensional simplicial complex $\Delta$ whose vertex set can be partitioned into $V_1, \ldots, V_p$ such that every face of $\Delta$ contains at most one element of each of the $V_i$ is called \textit{balanced} (the definition of a balanced complex usually requires that all maximal faces have the same cardinality; we do not make this assumption).  By construction, $\K'$ is balanced.

We check that if $\Delta$ is balanced and $(p-1)$-dimensional, then $\g_{p-1}(\Delta) \leq 2p$ only if there is a set of vertices $W=\cup_{i=1}^p \{v_{i,0}, v_{i,1}\}$ with $\{v_{i,0}, v_{i,1}\} \subseteq V_i$ such that $\{v_{1,j_i},v_{2,j_2},\ldots,v_{p,j_p}\} \in \Delta$ for all $(j_1,\ldots,j_p) \in \{0,1\}^p$.  In this case, $\Delta[W]$ is the boundary of a $p$-dimensional cross-polytope.  Suppose that $\g_{p-1}(\Delta) \leq 2p$.  Let $W \subset V$ be of minimal size such that $\tilde{H}_{p-1}(\Delta[W];\field) \neq 0$.  If for some $i$, $W \cap V_i$ is a single vertex $\{v\}$, then all $(p-1)$-faces of $\Delta[W]$ contain $v$ and therefore $\tilde{H}_{p-1}(\Delta[W];\field) = 0$.  We conclude that $|W \cap V_i| = 2$ for all $i$, and therefore $\K[W]$ is contained in the boundary of a $p$-dimensional cross-polytope.  The claim then follows.

If $v_{i,0},v_{i,1} \in V_i$ for $1 \leq i \leq p$, then the probability that $\K'[v_{1,0},v_{1,1},\ldots,v_{p,0},v_{p,1}]$ is the boundary of a $p$-dimensional cross-polytope is $a^{2^p}$.  Hence the expected number of such boundaries in $\K'$, denoted $E(\cp_p(\K))$, is ${n/p \choose 2}^{p}a^{2^p}$.  Also, the expected number of $(p-1)$-faces of $\K$ is $E(f_{p-1}(\K)) = a(n/p)^p$.  By linearity of expectation, $$E(f_{p-1}(\K)-\cp_p(\K)) = a(n/p)^p-{n/p \choose 2}^{p}a^{2^p}.$$ By choosing $a = C'_pn^{-\frac{p}{2^p-1}}$ for $C'_p$ sufficiently small and independent of $n$, there exists $\K'$ such that $f_{p-1}(\K)-\cp_p(\K) \geq C_p n^{p-\frac{p}{2^p-1}}$ for some $C_p$ independent of $n$.

Let $F_1, \ldots, F_{\cp_p(\K')}$ be a collection (possibly containing duplicates) of $(p-1)$-faces of $\K'$ such that every boundary of a $p$-dimensional cross-polytope in $\K$ contains some $F_i$.  Construct $\K$ from $\K'$ by removing all the $F_i$.  Then $\g_{p-1}(\K) > 2p$ and $f_{p-1}(\K) \geq C_p n^{p-\frac{p}{2^p-1}}$.
\end{proof}

We also consider the existence of $2$-dimensional simplicial complexes with both $f_2$ and $\g_2$ large.  For that we first need a technical lemma.
\begin{lemma}
\label{MinSizeHom}
Suppose that $\K$ is a balanced, two-dimensional simplicial complex.  Let $W \subset V(\K)$ be a minimal set such that $\tilde{H}_2(\K[W];\field) \neq 0$, i.e. if $W' \subsetneq W$, then $\tilde{H}_2(\K[W'];\field) = 0$.  Then $f_2(\K[W]) \geq 2|W|-4$.
\end{lemma}
\begin{proof}
We partition $W$ into $X,Y,Z$ so that every face of $\K[W]$ contains at most one vertex from each of $X,Y,Z$.  If $F$ is a face of $\K[W]$, construct $\K[W]-F$ by removing $F$ and all faces that contain $F$ from $\K[W]$.  If $\K[W]$ contains a face $F$ such that $\tilde{H}_2(\K[W]-F;\field) \neq 0$, then it suffices to prove that $f_2(\K[W]-F) \geq 2|W|-4$. Hence, we may assume without loss of generality that if any face of $\K[W]$ is removed to create $\Delta$, then $\tilde{H}_2(\Delta;\field)=0$.

$\K[X,Y]$ is a bipartite graph, say with $t$ edges.  If $xy$ is an edge in $\K[X,Y]$, then there are at least two vertices $z_1, z_2 \in Z$ such that $xyz_1, xyz_2 \in \K$; otherwise $xy$ could be removed without changing $\tilde{H}_2(\K[W];\field)$, a contradiction.  Therefore, it suffices to show that $t \geq |W|-2$.

First we show that $\K[X,Y]$ is connected.  Let $W_1, \ldots, W_s$ be the vertex sets of the connected components of $\K[X,Y]$.  $\K[W] = \cup_{j=1}^s \K[W_j,Z]$, and $\K[W_i,Z] \cap (\cup_{j=1}^{i-1} \K[W_j,Z]) = \K[Z]$, which is a set of isolated vertices and hence has vanishing first homology.  It follows from induction on $i$ and the Mayer-Vietoris sequence that $\tilde{H}_2(\K[W];\field) = \oplus_{j=1}^s \tilde{H}_2(\K[W_j,Z];\field)$.  Hence, by the minimality assumption, $s=1$ and $\K[X,Y]$ is connected.

It follows from the Euler-Poincar\'e formula that $\dim_\field(\tilde{H}_1(\K[X,Y];\field)) = t-|X|-|Y|+1$.  Hence, it suffices to show that $|Z| \leq \dim_\field(\tilde{H}_1(\K[X,Y];\field))+1$, which would then imply that $t \geq |X|+|Y|+|Z|-2 = |W|-2$.  Let $Z = \{z_1,\ldots,z_r\}$.  For $1 \leq i \leq r$, it must be that $\tilde{H}_1(\lk_{\K[W]}(z_i);\field) \neq 0$, or else $\tilde{H}_2(\K[W-\{z_i\}];\field) = \tilde{H}_2(\K[W];\field)$, a contradiction to the minimality assumption.  Let $b(i)$ be the dimension over $\field$ of the image of $\tilde{H}_1(\K[X,Y];\field)$ under the map on homology induced by the inclusion of $\K[X,Y]$ into $\K[X,Y,\{z_1,\ldots,z_i\}]$.  Consider a nonzero cycle $C \in \tilde{H}_1(\lk_{\K[W]}(z_i))$, and consider the Mayer-Vietoris sequence on $U = \K[X,Y,\{z_i\}]$ and $W = \K[X,Y,\{z_1,\ldots,z_{i-1}\}]$ with $U \cup W = \K[X,Y,\{z_1,\ldots,z_i\}]$ and $U \cap W = \K[X,Y]$.  If $C \neq 0$ in $\tilde{H}_1(\K[X,Y,\{z_1,\ldots,z_{i-1}\}])$, then $b(i) \leq b(i-1)-1$.  Otherwise, $\tilde{H}_2(\K[X,Y,\{z_1,\ldots,z_i\}];\field) \neq 0$, which implies that $i=r$.  Then $b(0) \geq r-1$, which proves the lemma.
\end{proof}

\begin{theorem}
There exists an absolute constant $C$ such that, for arbitrarily large $n$ and $k \leq n$, there exists a two-dimensional simplicial complex $\K$ with $n$ vertices satisfying $\g_2(\K) > k$ and $f_2(\K) \geq Cn^{5/2}k^{-3/2}$.
\end{theorem}
\begin{proof}
We may assume that $k \leq (2C+\mu)^{2/3}n^{1/3}$ for some fixed $\mu$, since there exists $\K$ with $\g_2(\K) = \infty$ and $f_2(\K) = {n-1 \choose 2}$.  Such a $\K$ can be constructed by taking the cone over a complete graph.

We use probabilitistic methods to construct an intermediate simplicial complex $\K'$ and then $\K$ with the claimed properties as follows.  Partition $V(\K) = V(\K')$ into $X,Y,$ and $Z$, each of size $n/3$, and let $0 < a < 1$ be a real number.  For all $x \in X, y \in Y, z \in Z$, $xy,xz,yz$ are edges in $\K'$, and $xyz$ is a face in $\K'$ with probability $a$, chosen independently of all other faces.  For all $W \subset V(\K')$ with $|W| \leq k$, let $\mathcal{T}^*_W$ be the set of all $(2|W|-4)$-subsets of $2$-faces of $\K'$ that are contained in $W$, and define $\mathcal{T}^* := \cup_{|W| \leq k}\mathcal{T}^*_W$.  Then define a function $T: \mathcal{T}^* \rightarrow \K'$ by choosing $T(\mathcal{T}) \in \mathcal{T}$ arbitrarily for all $\mathcal{T} \in \mathcal{T}^*$.  Construct $\K$ by deleting $T(\mathcal{T})$ from $\K'$ for all $\mathcal{T} \in \mathcal{T}^*$.

If $f_2(\K[W]) \geq 2|W|-4$ for some $|W| \leq k$, then let $\mathcal{T}$ be a set of $2|W|-4$ $2$-faces of $\K[W]$.  Then some face of $\mathcal{T}$ should have been deleted in the construction of $\K$, a contradiction.  Since $\K$ is balanced, and $f_2(\K[W]) < 2|W|-4$ for all $W$ with $|W| \leq k$, we conclude by Lemma \ref{MinSizeHom} that $\g_2(\K) > k$.  Next we show that by choosing $a = \epsilon n^{-1/2} k^{-3/2}$ for an appropriate value of $\epsilon$ independent of $n$ or $k$, $E(f_2(\K)) \geq Cn^{5/2}k^{-3/2}$.

For any value of $a$, $E(f_2(\K')) = a(n/3)^3 = an^3/27$.  Also, $f_2(\K) \geq f_2(\K')-|\mathcal{T}^*|$.  If we show that for all $a \leq \epsilon n^{-1/2} k^{-3/2}$, $E(|\mathcal{T}^*|) \leq an^3/54$, then it follows by linearity of expectation that $E(f_2(\K)) \geq an^3/54$.  Thus, there exists some $\K$ with $f_2(\K) \geq an^3/54$, which proves the theorem.

Let $t_i = |\{\mathcal{T} \in \mathcal{T}^*: |\mathcal{T}| = 2i-4\}|$.  Then $|\mathcal{T}^*| = \sum_{i=1}^k t_i$, and it suffices to show that if $a \leq \epsilon n^{-1/2} k^{-3/2}$, then $E(t_i) \leq C_0an^{5/2}$, for some absolute constant $C_0$, by $k < (2C+\mu)^{2/3}n^{1/3}$.  Since there are ${n \choose i}$ sets $W$ such that $|W|=i$, and for such a $W$ there are ${{i \choose 3} \choose 2i-4}$ sets $\mathcal{T}$ of size $(2i-4)$ of $3$-subsets of $W$, we have $$E(t_i) \leq a^{2i-4}{{i \choose 3} \choose 2i-4}{n \choose i}.$$

We need to verify that $$a^{2i-4}{{i \choose 3} \choose 2i-4}{n \choose i} \leq C_0an^{5/2},$$ which follows from $$a^{2i-4}{{i^3/6} \choose 2i-4}\frac{n^i}{i!} \leq C_0an^{5/2}.$$  By Stirling's approximation, this follows from $$a^{2i-5} \leq n^{-i+5/2}i^{(-6i+12)+(2i-4)+i}e^{-2i+4}e^{-i}6^{2i-4},$$ or for an appropriate constant $\epsilon$, $a \leq \epsilon n^{-1/2}i^{\frac{-3i+8}{2i-5}}$.  Since $i \leq k$, this follows from $a \leq \epsilon n^{-1/2} k^{-3/2}$, proving the result.
\end{proof}

The Ramanujan graphs of \cite{RGraphs} are examples of graphs with large girth and many edges.  The Ramanujan complexes of \cite{RComp} also have many faces and high girth, although under a definition of girth that is different from what we use.  Perhaps these constructions can be adapted to our setting to prove that the bounds of Conjecture \ref{GenConjecture} are, at least in some cases, tight.

\section{Connections with the multiplicity conjecture}
We conclude our study of Moore bounds by noting the connection with commutative algebra, and in particular the multiplicity conjecture.

Consider the polynomial ring $S$ over a field $\field$ generated by variables $x_1,\ldots,x_n$.  With every simplicial complex $\K$ we associate its \textit{Stanley-Reisner ideal} $I_{\K} \subset S$ generated by non-faces of $\K$: $I_\K := (\prod_{x_i \in L}x_i: L \subset V, L \not\in \K)$ (see \cite{St96}) and its \textit{Stanley-Reisner ring} $\field[\K] := S/I_\K$.

If $I$ is a graded ideal of $S$, then we construct a graded minimal free resolution of $S/I$ as an $S$-module.

\[ 0 \rightarrow
\bigoplus_{j \in \mathbb{Z}} S(-j)^{\beta_{l,j}} \rightarrow \ldots \rightarrow \bigoplus_{j \in \mathbb{Z}} 
S(-j)^{\beta_{1,j}} \rightarrow S \rightarrow S/I \rightarrow 0. \] 
In the above expression, $S(-j)$ denotes $S$ with grading shifted by $j$, and $l$ denotes the length of the resolution. In particular, $l \geq \codim(S/I)$.  The numbers $\beta_{i,j}$ are called the \textit{algebraic Betti numbers} of $I$.

We define a set of quantitites $\tilde{g}_{p-1}$ of a $(d-1)$-dimensional simplicial complex $\K$ in terms of the resolution.  First define the \textit{maximal shifts} of $\K$ as the largest indices of nonvanishing Betti numbers: $M_i(\K):=\max\{j : \beta_{i,j}\neq 0\}$.  The first $\codim(\field[\K])=n-d$ maximal shifts of $\K$ are strictly increasing, and so there are $d$ integers $1 \leq Q_0 < Q_1 < \ldots < Q_{d-1} \leq n$ that are not among the first $n-d$ maximal shifts of $\K$.  For $1 \leq p \leq d$, we define $\tilde{g}_{p-1}(\K) := Q_{p-1}(S/I_\K)+1$.

Our definition of the girths of a simplicial complex is closely related to $\tilde{g}$.  In the following lemma, we make use of Hochster's formula (see \cite[Theorem II.4.8]{St96}), which states that $$\beta_{i,j}(\field[\K]) = \sum_{|W|=j} \dim_{\field} \left(\tilde{H}_{|W|-i-1}(\K[W];\field) \right).$$

\begin{lemma}
\label{GirthLink}
Let $\K$ be a $(d-1)$-dimensional simplicial complex.  Then $\tilde{g}_{p-1}(\K) = \min \{n-d+p+1,\g_{p-1}(\K)\}$ for all $p$.
\end{lemma}
\begin{proof}
We show that $$\g_{p-1}(\K) = g'_{p-1}(\K) := \min \left\{|W|-j: \tilde{\beta}_{p-1+j}(\K[W]) > 0, j \geq 0 \right\}.$$  The result then follows by Hochster's formula.

First we prove that $g'_{p-1}(\K) \leq \g_{p-1}(\K)$.  Choose $W \subset V(\K)$ and $F \in \K$ so that $\tilde{H}_{p-1}(\lk_\K(F)[W];\field) \neq 0$.  We show that for some $F' \subseteq F$, $\tilde{H}_{p-1+|F'|}(\K[W \cup F'];\field) \neq 0$.  First consider the case that $F$ is a single vertex.  If $\tilde{H}_{p-1}(\K[W];\field) \neq 0$, then the inequality holds.  Otherwise, it follows from the Mayer-Vietoris sequence $$\ldots \rightarrow \tilde{H}_{p}(\K[W \cup F];\field) \rightarrow \tilde{H}_{p-1}(\lk_\K(F)[W];\field) \rightarrow \tilde{H}_{p-1}(\K[W];\field) \rightarrow \ldots$$ that $\tilde{H}_{p}(\K[W \cup F];\field) \neq 0$, and the inequality holds in this case as well.  Now suppose that $F$ contains several vertices, and let $v \in F$.  By induction on $|F|$, there exists $F' \subseteq F-\{v\}$ such that $\tilde{H}_{p-1+|F'|}(\lk_{\K}(v)[W \cup F'];\field) \neq 0$.  Now apply the previous argument to $\lk_\K(F'-\{v\})$.

Next we show that $g'_{p-1}(\K) \geq \g_{p-1}(\K)$.  Suppose that $\tilde{H}_{p-1+j}(\K[W];\field) \neq 0$ for some $|W|=g'_{p-1}(\K)+j$.  If $j=0$, then the claim is proven, so suppose that $j \geq 1$.  By defintion of $g'_{p-1}$, $W$ is of minimal size so that $\tilde{H}_{p-1+j}(\K[W];\field)$ is nonvanishing.  Hence by the Mayer-Vietoris sequence, for any vertex $v \in W$ it follows that $\tilde{H}_{p-2+j}(\lk_\K(v)[W-\{v\}];\field) \neq 0$.  Furthermore, by the definition of $g'_{p-1}$ and the argument of the previous paragraph, $W-\{v\}$ must be a minimal set with this property.  Hence by repeating this procedure $j$ times, there exists a face $F \subset W$ with $|F|=j$ such that $\tilde{H}_{p-1}(\lk_\K(F)[W-F];\field) \neq 0$.  This proves the lemma.
\end{proof}

The multiplicity conjecture is a prominent statement in commutative algebra.  Part of the statement places an upper bound on $f_{d-1}$ of a $(d-1)$-dimensional simplicial complex in terms of its maximal shifts, or equivalently, in terms of its girths.  In general, let $N$ be a graded module over $S$ with codimension $c$, multiplicity $e(N)$, and first $c$ maximal shifts $M_1,\ldots,M_c$.  Then $e(N) \leq M_1 \ldots M_c / c!$.  The conjecture was first posed in \cite{HerzSr98}, and it follows from the Boij-S\"{o}derberg conjecture \cite{BoijSo}.  The Boij-S\"{o}derberg conjecture was proven in \cite{EisSch} for the Cohen-Macaulay case, and generalized to the non-Cohen-Macaulay case in \cite{NonCMMC}.

For a $(d-1)$-dimensional simplicial complex $\K$ with $n$ vertices, $e(S/I_\K) = f_{d-1}(\K)$.  Therefore, in terms of girths, the multiplcity conjectures states that $$f_{d-1}(\K) \leq \frac{n(n-1)\ldots(n-d+1)}{(\g_0(\K)-1)\ldots (\g_{d-1}(\K)-1)}.$$ Although a general combinatorial proof of this result remains elusive, some papers such as \cite{NovSw} establish the result for some classes of simplicial complexes.  A simple proof for the one-dimensional case follows from the observation that if $\K$ has girth at least $g$ and $|W| = g-1$, then $\K[W]$ is a forest and has at most $g-2$ edges.  The result follows by adding over all such $W$.

The results in this paper are inspired by the observation that Theorem \ref{IrregularMoore} is generally much stronger than the multiplicity conjecture for the case of graphs.  We see that the bound of Conjecture \ref{GenConjecture} is generally much stronger than that of the multiplicity conjecture when $n$ is large and the girths are small.


\begin{thebibliography}{999}

\bibitem{MooreGraphs} N.~Alon, S.~Hoory, and N.~Linial, The Moore bound for irregular graphs, Graphs Combin. {\bf 18} (2002), 53--57.

\bibitem{Biggs} N.~Biggs, {\em Algebraic graph theory, 2nd edn.}, Cambridge University Press, Cambridge, 1993.

\bibitem{NonCMMC} M.~Boij, Betti numbers of graded modules and the Multiplicity Conjecture in the non-Cohen-Macaulay case, math arXiv: 0803.1645.

\bibitem{BoijSo} M.~Boij and J.~S\"{o}derberg, Graded Betti numbers of Cohen-Macaulay modules and the multiplicity conjecture, math arXiv: 0611081v2.

\bibitem{Bollobas} B.~Bollob\'{a}s, {\em Extremal graph theory}, Harcourt Brace Jovanovich Publishers, Academic Press, London, 1978.

\bibitem{WeakerMoore} R.~Dutton, D.~Brigham, R.C., Edges in graphs with large girth, Graphs Combin. {\bf 7} (1991), 315--321.

\bibitem{EisSch} D.~Eisenbud and F.~Schreyer, Betti numbers of graded modules and cohomology of vector bundles, math arXiv: 0712.1843v2.

\bibitem{Flag} A.~Frohmader, Face vectors of flag complexes, Israel J. Math. {\bf 164} (2008), 153--164.

\bibitem{HerzSr98} J.~Herzog and H.~Srinivasan, Bounds for multiplicities, {\em Trans.~Amer.~Math.~Soc.} {\bf 350} (1998),  2879--2902.

\bibitem{MooreSC} A.~Lubotzky, R.~Meshulam, A Moore bound for simplicial complexes, Bull. London Math. Soc. {\bf 39} (2007), 353--358.

\bibitem{RGraphs} A.~Lubotzky, R.~Phillips, P.~Sarnak, Ramanujan graphs, {\em Combinatorica} {\bf 8} (1988), 261--277.

\bibitem{RComp} A.~Lubotzky, B.~Samuels, U.~Vishne, Explicit constructions of Ramanujan complexes of type $A_d$, {\em European J. Combin.} {\bf 26} (2005), 965--993.

\bibitem{NovSw} I.~Novik and E.~Swartz, Face ring connectivity via CM-connectivity sequences, {\em Canad. J. Math}, to appear.

\bibitem{St96} R.~Stanley, {\em Combinatorics and Commutative Algebra}, Second Edition, Birkh{\"a}user, Boston, 1996.

\end{thebibliography}
\end{document}